\documentclass[preprint]{elsarticle}

\usepackage{xcolor}
 \usepackage{amsthm}
\usepackage{soul}
\usepackage{comment}
\usepackage{amsmath,amssymb}
\usepackage[cal=boondox,scr=boondoxo]{mathalfa}
\usepackage{bigints}
\usepackage{todonotes}
\presetkeys{todonotes}{color=orange!30}{}
\usepackage{url}
\def\mm{\textcolor{magenta}} 

\def \CC{{\bold C}}
\usepackage{amsmath}

\newtheorem{theo}{Theorem}[section]
\usepackage{yfonts}

\newtheorem{lemma}{Lemma}[section]

  \newtheorem{remark}{Remark}[section]
\newtheorem{definition}{Definition}[section]

\DeclareMathAlphabet{\mathpzc}{OT1}{pzc}{m}{it} 

\def \M{\mathcal M}
\def \W{\mathcal W}
\def \K{\mathcal K}

\def\C{\mathcal{C}}
\def\U{\mathcal{U}}
\def \cR{\mathbb R}
\def \rr{\mathbb R}

\def \cN{\mathbb N}

\def \cS{\mathcal S}
\def \N{\mathcal N}

\def \B{\mathbb B}
\def \eps{\varepsilon}

\def \h{\Psi}

\def\vz{\mathpzc{v}}
\def\d{ {\rm d} }

\def \vsm{\vskip 0.2 truecm}
\def \ds{\displaystyle}

\journal{Systems and Control Letters}
\begin{document}
\begin{frontmatter}

\title{ 
Higher-Order Normality and No-Gap Conditions in Impulsive Control with 
 $L^1$-Control Topology
  \tnoteref{label1}}
\tnotetext[label1]{This research is partially supported by INdAM-GNAMPA Project 2024, CUP E53C2300167000, and by PRIN 2022, Prot. 2022238YY5, CUP C53D23002370006.}
\author[1]{Monica Motta\corref{cor1}}
\ead{motta@math.unipd.it}

\author[2]{Michele Palladino}
\ead{michele.palladino@univaq.it}

\author[1]{Franco Rampazzo}
\ead{rampazzo@math.unipd.it}

\cortext[cor1]{Corresponding author.}

\address[1]{University of Padova, Department of Mathematics "Tullio Levi-Civita",
via Trieste 63, Padova, 35121, Italy}
\address[2]{University of  L'Aquila, Department of Engineering, Information Sciences and Mathematics, via Vetoio, 67100 L'Aquila, Italy }

 \date{\today}

%
%
%
%
%

\begin{abstract}
In optimal control, extending the class of admissible controls is a common strategy to guarantee the existence of optimal solutions. However, such extensions may introduce a gap between the infimum of the original problem and the minimum of the extended one, especially in the presence of endpoint constraints. Since Warga's seminal work, normality of first-order necessary conditions for extended minimizers has been recognized as a sufficient condition to avoid this phenomenon, though it is far from being necessary.
In this paper, we consider impulsive extensions of control-affine systems with unbounded controls. We establish that a notion of \textit{higher-order normality}, based on iterated Lie brackets of the system’s vector fields, suffices to prevent an infimum gap. The key novelty of this manuscript consists in showing that this holds under a local topology defined by the $L^1$-distance between controls, rather than the more common $L^\infty$-distance between trajectories. Among the reasons that motivate  the interest in this issue, let us mention that a counterexample by R. B. Vinter shows that for a different extension --based on convexification of the velocity set-- a local extended minimizer that is normal with respect to the $L^1$-norm  of the controls may still exhibit a gap. 
Our method relies on set-separation techniques. Such an approach makes it possible to derive higher-order  conditions and to exploit the corresponding notion of higher-order normality.
\end{abstract}



\end{frontmatter}

\section{Introduction} 


When an optimal control problem fails to have a local minimizer, a common strategy consists  of: 
\begin{itemize}
	 \item[i)] embedding the optimal control problem of interest --which, from now on we will refer to  as the {\it 
original optimal control problem}-- in a larger one  --which we will refer to as the {\it extended optimal control}-- that possesses a local minimizer,
\item[ii)] approximating the minimizer of the extended optimal control problem with solutions of the original optimal control problem.
\end{itemize}

Such a procedure can be successfully carried out  as long as an ``infimum gap" does not occur, that is as long as the infimum value of the original optimal control problem coincides with the minimum value of the extended optimal control problem. Therefore, it is of interest to provide conditions guaranteeing that the optimal control problem under consideration does not have an infimum gap.

 Some early results by J.Warga  have inspired the idea that,  under certain hypothesis, a general criterion for the absence of an infimum gap can be stated as follows:
\begin{equation}\label{ass_intro}
\begin{array}{ll}
\textit{If a set of necessary conditions holds true in normal form}, \\ 
\textit{then the infimum gap does not occur."}  
\end{array}
\end{equation}

The expression ``in normal form" here means that the corresponding cost multiplier is  different from zero for each set of multipliers. As a simple consequence of statement \eqref{ass_intro}, any optimal control problem without end-point constraints and state constraints does not have an infimum gap.

In fact, it turns out that the validity of statement \eqref{ass_intro} heavily depends on several aspects, such as:
\begin{itemize}
\item[i)] the type of dynamics appearing in the optimal control problem (e.g., multifunctions, nonlinear control differential equations, control-affine nonlinear differential equations, etc.);
\item[ii)] the type of extension one is carrying out on the original optimal control problem (relaxation, impulsive extension, etc.);
\item[iii)] the topology that one considers in the notion of local minimizer ($L^\infty$, $W^{1,1}$, etc.);
\item[iv)] the type of necessary conditions taken into account.
\end{itemize}

As suggested above, the first result making statement \eqref{ass_intro} precise and effective was obtained by J. Warga in \cite{warga2}, \cite{warga}, \cite{warga1}: in these works, the controlled differential equations are extended by relaxation, namely by taking the convex hull of the controlled dynamics. Other results in this direction were obtained in \cite{PR}, \cite{PV2}, and \cite{FM222}, \cite{FM322} (in the last two mentioned papers, the focus was particularly placed on problems with state constraints and on non-degenerate necessary conditions).

Let us highlight, at this stage, that all of these results hold true for local minimizers with respect to the $L^\infty$-metric, the latter being the one compatible with the well-known Relaxation Theorem \cite{warga}.

The problem of extending statement \eqref{ass_intro} to the case in which the necessary conditions are applied to a $W^{1,1}$-local minimizer was posed for the first time in \cite{ioffe} for the case in which the control dynamics of the optimal control problem takes the form of a non-convex differential inclusion and the necessary condition's adjoint equation is expressed in terms of the "partially convexified Euler-Lagrange equation" (notice that, in \cite{ioffe}, the extension considered is still the standard relaxation extension). However, a counter-example to this conjecture (expressed in terms of a non-convex differential inclusion) was presented in \cite{vinter14}; furthermore, it was proved in \cite{PV1} that statement \eqref{ass_intro} holds true for $L^\infty$-local minimizers when the necessary condition's adjoint equation is expressed in terms of the "fully convexified Hamiltonian Inclusion".

Another common extension is the so-called impulsive extension (see e.g.,   \cite{AKP1}, \cite{BR}, \cite{MR}). This case has been mainly investigated when the dynamics is (non-linear and) control-affine, and the controls are unbounded. Roughly speaking, the original control system is embedded --through time-reparameterization techniques or a suitable use of measures-- into a new one where the state is allowed to evolve instantaneously. For the impulsive extension, the first result showing the validity of assertion \eqref{ass_intro} was proved in \cite{MRV} for a notion of $L^\infty$-local minimizer and a set of necessary conditions expressed in terms of the Pontryagin Maximum Principle. An improvement of this result for the case in which a state constraint is added to the problem formulation was obtained in \cite{FM121}. On the other hand, it is interesting to observe that the result in \cite{FM124} is the first one to show that assertion \eqref{ass_intro} holds true also for the case in which one deals with a $W^{1,1}$-local minimizer, when the impulsive extension is taken into account. (Actually, the type of extension discussed in \cite{FM124} is more general than the impulsive extension; however, it is not general enough to encompass the relaxation extension.) Still, for the impulsive extension and for $L^\infty$-local minimizers, it was shown in \cite{MPR} that statement \eqref{ass_intro} holds true even when one considers higher-order necessary conditions expressed in terms of (iterated) Lie brackets. Notice that, while most of the results mentioned so far (with the exception of \cite{PR}) are proved using perturbation arguments based on Ekeland's Variational Principle (or on Stegall's Variational Principle), the higher-order no-gap condition proved in \cite{MPR} is based on a set-separation approach following some general ideas from \cite{sus} and \cite{Suss1}.

The main goal of this article is to show that the results proved in \cite{MPR} can be extended to the case in which one replaces the "strong" notion of $L^\infty$-local minimizer with the "weak" notion of $W^{1,1}$-local minimizer. This result will be shown by using a set-separation approach, and by carefully adapting the analysis carried out in \cite{MPR}. The crucial aspect to address is to show that the optimal, extended -sense control and the approximating, original controls remain "sufficiently close" with respect to the $L^1$-norm of the controls. These technical aspects are summarized in Lemmas \ref{L0}, \ref{L1}, and \ref{L2}.

\subsection{Notations}{We now summarize the main notations that we will use throughout the paper.} The state space is modeled as a Riemannian manifold $(\M, \langle\cdot,\cdot\rangle)$, and the endpoint constraint (or \textit{target}) is defined as a closed subset $\cS \subseteq [0,+\infty) \times \M$. We consider a \textit{cost functional} $\h : [0,+\infty) \times \M \to \cR$ and impose an \textit{energy constraint} given by an upper bound $K > 0$ (possibly $K = +\infty$).

The set $\U$ of \textit{strict-sense controls} is defined by
\[
\U := \bigcup_{T > 0} \left( \{T\} \times L^1\big([0,T], \C \big) \right),
\]
where $\C \subseteq \cR^m$ is a closed cone. For any control pair $(T,u) \in \U$, we define a \textit{strict-sense process} as a  4-tuple $(T, u, x, \vz)$, where $(x, \vz)$ solves the Carathéodory system
\begin{equation} \label{E2}
\left\{
\begin{aligned}
\frac{dx}{dt}(t) &= f(x(t)) + \sum_{i=1}^m g_i(x(t)) u^i(t), \\
\frac{d\vz}{dt}(t) &= |u(t)|, \\
(x, \vz)(0) &= (\check{x}, 0),
\end{aligned}
\right.
\end{equation}
with given smooth vector fields $f$, $g_1, \dots, g_m$. A strict-sense process is said to be \textit{feasible} if it satisfies the endpoint condition $(T, x(T)) \in \cS$ and the energy constraint $\vz(T) \le K$. Observe that $\vz(t)$ represents the $L^1$-norm of the control $u$ on the interval $[0,t]$. Let $\mathcal{X}$  and  $\mathcal{X}^\cS$
denote the set of  strict-sense processes and the subset of the feasible ones, respectively. 

The \textit{original optimal control problem} is then posed as:
\begin{equation}\tag{$\mathcal{P}$}\label{P}
\inf \left\{ \h(T, x(T)) : (T,u,x,\vz)\in\mathcal{X}^\cS \right\}.
\end{equation}

Due to the lack of control bounds and absence of growth conditions, minimizing sequences may fail to converge within the space of admissible trajectories. To address this, we consider an \textit{extended formulation} of the problem following the graph-completion approach (see, e.g., \cite{BR, MR, MiRu, GS, KDPS, AKP1, WZ}), 
where $(S, w^0, w)$ lies in the space of \textit{extended controls}
\[
\mathcal{W} := \bigcup_{S>0} \left( \{S\} \times L^\infty([0, S], \CC)\right) .
\]
Here, $\CC$ is the closure of the set
$\CC^+:=\{(w^0,w)\in]0,+\infty[ \times \C: \,  w^0+|w|=1\}$,
	and $(y^0, y, \beta)$ solves the extended control system
\begin{equation} \label{extended}
\left\{
\begin{aligned}
\frac{d y^0}{ds}(s) &= w^0(s), \\
\frac{d y}{ds}(s) &= f(y(s)) w^0(s) + \sum_{i=1}^m g_i(y(s)) w^i(s), \\
\frac{d \beta}{ds}(s) &= |w(s)|, \\
(y^0, y, &\beta)(0) = (0, \check{x}, 0).
\end{aligned}
\right.
\end{equation}
As in the case of strict-sense processes, an  {\em extended process} $(S, w^0, w, y^0, y, \beta)$ is called \textit{feasible} if $(y^0(S), y(S), \beta(S)) \in \cS \times [0,K]$. Let  $X_\mathcal{W}$ be the set of extended processes and write  $X^{\cS}_\mathcal{W}$ for the subset of the feasible ones. Then, the  \textit{extended optimal control problem} is
 \begin{equation}\tag{P{\scriptsize e}}\label{Pe}
\inf \left\{ \h(y^0(S), y(S)) : (S, w^0, w, y^0, y, \beta)\in X^{\cS}_\mathcal{W} \right\}.
\end{equation}
The variable $y^0$ can be interpreted as a reparametrization of time. In this setting, every strict-sense process corresponds to an extended process  with  controls in the subset
\[
\mathcal{W}^+ := \bigcup_{S>0} \left( \{S\} \times L^\infty([0, S], \CC^+)\right) ,
\]
We refer to such processes as \textit{embedded strict-sense processes}:  $X_{\mathcal{W}^+}$ and $X^{\cS}_{\mathcal{W}^+}$ will denote the set of  embedded strict-sense processes and the subset of the feasible ones, respectively. Given a strict-sense process $(T, u, x, \vz)$, one defines the time change $\sigma(t) := t + \vz(t)$, and its inverse $y^0 := \sigma^{-1}$, yielding  
\begin{equation}\label{seq}
(S, w^0, w, y^0, y, \beta) := \left( \sigma(T), \frac{dy^0}{ds}, (u \circ y^0) \cdot \frac{dy^0}{ds}, y^0, (x ,\vz) \circ y^0 \right)
\end{equation}
in $X_{\mathcal{W}^+}$. Conversely, one recovers a strict-sense process from an embedded one by setting $\sigma := (y^0)^{-1}$. This gives a bijective correspondence between strict-sense processes and their embedded counterparts, preserving feasibility and cost values.

As such, the original problem is equivalent to the following embedded formulation:
\begin{equation}\tag{P}\label{Ps}
\inf \left\{ \h(y^0(S), y(S)) : (S, w^0, w, y^0, y, \beta)\in X^{\cS}_{\mathcal{W}^+}\right\}.
\end{equation}
 The \textit{impulsive extension} of the problem is obtained  by permitting $w^0$ to vanish on non-degenerated sub-intervals $I$ of  $[0,S]$. In such cases, the state variable $y$ describes instantaneous changes at time values $t = y^0(I)$ for intervals $I \subseteq [0,S]$ where $w^0 = 0$. (This can also be equivalently described in terms of trajectories of bounded variation, as in \cite{KDPS, AKP1, FM224}).

\medskip

We consider a notion of {\em local} infimum gap with respect to the following $L^1$-control distance:  
\begin{equation}\label{dL^1}
				\d\big(z_1,z_2\big):=    |S_1-S_2|+ \|(w^0_1, w_1)-(w^0_2, w_2)   \|_{L^1[0,S_1\wedge S_2]},
		\end{equation} 
	  $z_i:=(S_i, w^0_i, w_i , y^0_i,y_i,\beta_i)\in X_\mathcal{W}$, $i=1,2$.
		Specifically, 	
		\begin{definition}\label{Dgap} Given   a  feasible extended   process $ \hat z =(\hat S,\hat w^0,\hat w,\hat y^0,\hat y,\hat \beta)$,  we say that  {\em at $\hat z$  there is a local  infimum  gap} if  
		\begin{equation}\label{igc} 
		\h\big((\hat{y}^0,\hat{y})(\hat{S}) \big) <\inf_{\left\{z \in X^{\cS}_{\mathcal{W}^+}: \ \d(z,\hat z)<r\right\} }\h\big((y^0,y)(S)\big)  
		\end{equation} 
for some $r>0$.
\end{definition}

 
\section{Some preliminaries}\label{Sprel} 
\subsection{Lie brackets} 
	If $h$, $k$ are $C^1$ vector fields on a differential manifold $\N$,\footnote{The regularity of the differential manifold $\N$ will be always assumed such that all considered brackets can be classically defined. For simplicity, one can  assume that  $\N$ is of class $C^\infty$.}  the {\it Lie bracket of $h_1$ and $h_2$}, on any chart, is
	$$
	[h,k](x) := Dk(x)\cdot h(x) -  D h(x)\cdot k(x) .
	$$ 
	If the vector fields are sufficiently regular,  one can iterate the bracketing process   and define, for instance, given a $5$-tuple ${\bf h}:=(h_1,h_2,h_3,h_4,h_5)$ of vector fields,  the bracket $[[h_1,h_2],h_3]$.   Accordingly, one can consider the  {\it  formal bracket } $B:=[[X_2,X_3],X_4]$ and set  $B({\bf h}) := [[h_2,h_3],h_4]$. 
	 
	\begin{definition}    Fix $k\in\cN$. If $\mu\geq 0$,  $r\ge1$, and $\nu\geq \mu+r$ are integers,  $B = B(X_{\mu+1},\ldots,X_{\mu+r})$ is an iterated  formal bracket, and   ${\bf h}=(h_1,\ldots,h_\nu )$  is a string of vector fields, we say that {\rm  ${\bf h}$ is of class $C^{B+k}$} if there is a $\nu$-tuple $(j_1,\ldots,j_\nu)\in\cN^\nu$  such that
		$h_i$ is of class $C^{j_i}$ for any $i=1,\ldots,\nu$  and 
		$B({\bf h})$ is a vector field of class $C^k$. In this case, we   call $(B,\mathbf{h})$   an {\em admissible $C^k$ bracket pair.} 
	\end{definition}
	\begin{definition} For every   $k\in\cN$, 
		we  use $\B^k$ to denote the (possibly empty)  set of  
		admissible $C^k$  bracket pairs  $(B,\mathbf{h})$,  such that  $\mathbf{h}:=(h_1,\ldots,h_\nu)$   is a $\nu$-tuple of vector fields  $h_j\in \{g_1,\ldots,g_{m_1}\}$ for every $j=1,\ldots,\nu$.   \footnote{For  more  detailed  definitions we refer to \cite{FeleqiRampazzo2017}.}
	\end{definition} 

\subsection{QDQ approximating cones}
We recall the definitions of {\it Quasi Differential Quotient  } (QDQ) of a set-valued map  and   of {\it QDQ-approximating cone}   introduced in \cite{PR} (see also \cite{MPR}).
In the following, we  set  $\cR_+:=[0,+\infty[$ and $\cR_-:=]-\infty,0]$. A  monotone nondecreasing function   $\rho : \cR_+\to\cR_+$ is  a   {\it modulus}   if 
 	$\ds\lim_{s\to 0^+}\rho(s) =\rho(0)= 0$. For any $\delta>0$ and $\xi\in\cR^N$, $B_\delta(\xi)$  is the open ball of center $\xi$ and radius $\delta$ in $\cR^N$. We use $Lin(\cR^N, \cR^n)$ to denote the set of linear maps from $\cR^N$ to $\cR^n$.
	\begin{definition}
		\label{qdq}
		Let $G : \cR^N \rightsquigarrow \cR^n$     be a 
		set-valued map, $(\bar  \xi,\bar y) \in \cR^N\times\cR^n  $,  let $\Lambda\subset Lin(\cR^N, \cR^n)$  be  a compact set,  and let $\Gamma\subset\cR^N$  be  any  subset.
		We
		say that $\Lambda$ is a {\rm Quasi Differential Quotient  (QDQ) of $G$ at  $(\bar  \xi,\bar y)$  in the direction of $\Gamma$}  if   there
		exist  a modulus $\rho$ and some $\bar\delta>0$ enjoying  the property that,
		for any $\delta\in[0,\bar\delta]$,  there is a continuous map  
		$(L_\delta,h_\delta):   B_{\delta}(\bar{ \xi})\cap\Gamma  \to Lin(\cR^N, \cR^n) \times \cR^n$ 
		such that, for all $ \xi
		\in   B_\delta(\bar{ \xi})\cap \Gamma$,\footnote{Notice that we are not assuming that $\bar \xi\in \Gamma$, so that it may well happen that  $\bar y\notin G(\bar\xi)$.}
		$$
		\begin{array}{l}
		\bar y +  L_\delta ( \xi)\cdot( \xi-\bar  \xi)  + h_\delta( \xi)\in G( \xi), \quad  |h_\delta( \xi)|\leq \delta \rho(\delta), \\
		 \min_{L'\in\Lambda}|L_\delta( \xi) - L'|\leq \rho(\delta). 
		\end{array}
		$$
	\end{definition}
	The notion of QDQ extends  to differential manifolds:
	\begin{definition}
	\label{agdqM}  
		Let $\mathcal{N}$, $\mathcal{M}$ be differential  manifolds of class  $C^1$.  Assume that $\tilde G : \mathcal{N} \rightsquigarrow \M$       is a 
		set-valued map, $(\bar \eps,\bar y) \in \mathcal{N}\times\M  $,  $\tilde\Lambda\subset Lin(T_{\bar\eps}\mathcal{N}, T_{\bar y}\M)$   is a compact set,  and $\tilde\Gamma\subset\mathcal{N}$  is any  subset.
		Moreover, let $\phi:U\to \cR^N$ and $\psi:V\to\cR^n$ be charts defined on  neighborhoods  $U$ and $V$ of $\bar\eps$ and $\bar{y}$, respectively, and assume that $\phi(\bar\eps) = 0$, $\psi(\bar y)=0$. Let $G:\cR^N\to\cR^n$ any extension of the map 
		$\psi\circ \tilde G\circ\phi^{-1}:  \phi(U)\to\cR^n$.
		We say that $\tilde\Lambda$ is   a {\em Quasi Differential Quotient (QDQ) of $\tilde G$ at  $(\bar \eps,\bar y)$  in the direction of $\tilde\Gamma$}  if  $\Lambda:= D\psi(\bar y) \cdot \Lambda\cdot D\phi^{-1}(0)$    is a  QDQ  of $G$ at  $(0,0)$ in the direction of $\Gamma:=\phi(\tilde\Gamma\cap U)$. 
	\end{definition}
	\begin{remark} {\rm By means of the concept of QDQ, introduced in \cite{PR} as a special case of Sussmann's AGDQ \cite{sus}, one obtains  a standard, non-punctured, Open Mapping Theorem. This key property, not shared by general AGDQs, underpins the set separation argument crucial to the proof of Theorem 3.1. For a more in-depth discussion on the properties of QDQs, we refer to \cite{PR}.}
\end{remark}
	\begin{definition}\label{ApprCone}    Let $\M$ be a $C^1$ differential manifold, $\mathcal E\subset\M$   a set.  A  {\em  QDQ approximating cone}  to $\mathcal E$ at $z\in \mathcal E$ is a
		convex cone $\K\subseteq T_{z}\M$ such that there exist  $N\in\cN$, a set-valued
		map  $G : \cR^N  \rightsquigarrow \M$, a convex cone $\Gamma\subset\cR^N$, and a  QDQ  $\Lambda:=\{L\}\subset \text{Lin}(\cR^N,\cR^n)$ of $G$  at  $(\xi,z)$ in the direction of $\Gamma$ such that $G(\xi+\Gamma)\subset \mathcal E$ and $\K =
		L\cdot\Gamma$. 
		We say that   {\em   $(G,\Gamma,\Lambda)$  generates the  cone $\K$}.
		If such triple $(G,\Gamma,\Lambda)$ satisfies
		$G(\xi+\Gamma)\subset \mathcal E\backslash \{z\}$, then we say that {\em the QDQ approximating  cone  $\K$
			is  $z$-ignoring}. 
	\end{definition}
	\begin{remark}{\rm The  classical  Boltyanski approximating cone  is a  special case of QDQ approximating cone (see, e.g.,  \cite{APR25}). Other approximating cones or spaces  may happen to be QDQ approximating  cones provided  some regularity assumptions   are satisfied (see \cite{MPRArXiv} for the case of the Clarke tangent cone).	}
	\end{remark}
%

\section{Main results}
We assume the following hypotheses:
\begin{itemize}
    \item[\rm (i)] The vector fields $f$, $g_1, \dots, g_m$ are of class $C^1$ and uniformly bounded along with their derivatives.
    \item[\rm (ii)] The final cost function $\h : \cR \times \M \to \cR$ is of class $C^1$.
    \item[\rm (iii)] The control set $\C \subseteq \cR^m$ is a closed cone of the form $\C = \C_1 \times \C_2$, with  $m_1, m_2\in\cN$ and $m_1 + m_2 = m$, where:  if $m_1 \ge 1$, the cone $\C_1 \subseteq \cR^{m_1}$ contains all lines $\{ r \mathbf{e}_i : r \in \cR \}$ for $i = 1, \dots, m_1$;\footnote{For any  $i=1,\dots,m$, $\mathbf{e}_i$ is the $i$-th element of the canonical basis of $\cR^m$.} the cone $\C_2 \subset \cR^{m_2}$ contains no lines.\footnote{This assumption is not restrictive up to a linear change of coordinates.}
\end{itemize}
We define the  {\it unmaximized Hamiltonian} $H: T^*\M\times \cR\times\cR\times\cR_+\times\C   \to \cR$ as 
$$
H(x,p,p_0,\pi,w^0,w ):= 
p_0w^0 + p\cdot\Big(f (x) w^0 +  \sum_{i=1}^{m}  g_{i}(x) w^{i}\Big) + \pi | w|,
$$ 
and  introduce  the following notion of higher-order extremal.
\begin{definition}\label{HOExtremal} Take $(\bar S,\bar w^0,\bar w , \bar y^0,\bar  y,\bar\beta)\in X^{\cS}_\mathcal{W}$.
		 Let  $\K$  be a QDQ   approximating cone to  the target $\cS$ at $(\bar y^0(\bar S),\bar y(\bar S))$. 
		We say that  $(\bar S,\bar w^0,\bar w , \bar y^0,\bar  y,\bar\beta)$ is  a {\em higher-order  $\h$-extremal} with respect to $\K$    if  there exist  a lift $(\bar y, p)\in AC([0,\bar S], T^*\M)$ and multipliers   $(p_0,\pi,\lambda)\in \cR\times \cR_-
		\times \cR_+$  such that  conditions {\rm(i)-(vi)} below are valid.
		\begin{itemize}
			\item[{\rm (i)}]  {\sc (non-triviality)} The triple $(p_0, p , \lambda)$ is non trivial, i.e.
			\begin{equation} 
				\label{fe1}
				(p_0, p , \lambda) \not= (0, 0,0) \,.
			\end{equation}
			Furthermore, if 
			$ \bar y^0(\bar S)>0$, then \eqref{fe1} can be strengthened  to 
			\begin{equation}\label{strongfe1}
				(p  , \lambda) \not= (0,0). 
			\end{equation}
			\item[{\rm (ii)}] {\sc (non-tranversality)}  If $J_K:=\{0\}$ for $\bar\beta(\bar S) < K $, and $J_K:=[0,+\infty)$ 
			for $\bar\beta(\bar S) = K$,
			\begin{equation}
				\label{fe4}
				(p_0,p(\bar S),\pi) \in \left[-\lambda D\Psi\big( (\bar y^0(\bar S),\bar y(\bar S))\big)- \K^\perp\right]\times J_K.
			\end{equation}
			 In particular,   
			\begin{equation}\label{piestzero}\pi = 0 \quad\hbox{provided}\quad \bar\beta(\bar S)< K.  \end{equation} 
			\item[{\rm (iii)}] {\sc (Hamiltonian equations)}  The path $ (\bar y, p) $ verifies 
			\begin{equation}
				\label{fe2def}
				\displaystyle  \frac{d}{ds}(\bar y,p) (s)\,=\, {\bf X}_{\bar H}\left(s,\bar y(s), p(s)\right) \quad\text{for a.e. $s\in [0,\bar S]$,}
			\end{equation} 
			where $\bar H=\bar H(s,y,p):= H\big(y,p, p_0,\pi,\bar w^0(s),\bar w(s)\big)$, and ${\bf X}_{\bar H}$ denotes   the ($s$-dependent) Hamiltonian vector field corresponding to ${\bar H}$.\footnote{If $K=K(s,y,p)$ is a differentiable map on the cotangent bundle  $T^*\M$, in any local system of canonical coordinates $({\mathfrak y },{\mathfrak p })$, the Hamiltonian vector field ${\bf X}_{K}$ corresponding to $K$ is defined as   
			 $\ds {\bf X}_{K}(s,y,p) := \left(D_{p}{K}, - D_{y}{K}\right)(s,y,p)$,
			  so that \eqref{fe2def} coincides with the extended system coupled with the  usual adjoint equation.}
			\item[{\em (iv)}] {\sc (First order maximization)} For a.e. $s\in [0,\bar S]$, 
			\begin{equation}\label{fe3}
					H\Big(\bar y(s), p(s),p_0 , \pi,\bar w^0(s),\bar w(s) \Big)= 
				\max_{(w^0,w )\in \CC }  H(\bar y(s), p(s), p_0 ,\pi,w^0,w ).
			\end{equation}
				\item[{\em (v)}] {\sc (Vanishing of the Hamiltonian)} 
			\begin{equation}\label{engine}
			\max_{(w^0,w  )\in \CC  }  H(\bar y(s), p(s), p_0 ,\pi,w^0,w )=0, \quad \text{for all }s\in[0,\bar S],
			 \end{equation}
			  {which, in case $\bar\beta(S)<K$, yields}
			\begin{gather}
				p(s)\cdot g_i(\bar y(s))=0,   \qquad \text{for all } s\in[0,\bar S],\   i=1,\dots,m_1.  \label{pg0hi1}
			\end{gather}
			\item[{\em (vi)}] {\sc (Higher-order conditions)} If  $\bar\beta(S)<K$ and  $(B,\mathbf{h})\in \B^0$,  
			\begin{gather}
				p(s)\cdot  B(\mathbf{h})(\bar y(s))=0,   \quad \text{for all } s\in [0,\bar S]. \label{pg000hi}
			\end{gather}
			Furthermore, if  $(B,\mathbf{h})\in \B^1$,     for  a.e.   $s\in [0,S]$  one has\footnote{ When $m=m_1=1$ and $g=g_1$,  equality \eqref{MP111new} reduces to the Legendre-Clebsch-type  condition 
				$  p(s)\,\cdot\, \big[{f }, g \big](\bar y(s))\,\bar w^0(s) =0$.
			 } 
			\begin{equation}\label{MP111new}
			p(s)\,\cdot\, \bigg(\big[{f  }, B(\mathbf{h})\big](\bar y(s))\,\bar w^0(s) + \ds\sum_{j=m_1+1}^{m}
			\big[g_j,B(\mathbf{h})\big](\bar y(s))\,\bar w^j(s)\bigg) =0.
			 \end{equation} 
		\end{itemize}
	\end{definition}

	\begin{definition}\label{HONormal} 
		Let $(\bar S,\bar w^0,\bar w , \bar\alpha,  \bar y^0,\bar  y,\bar\beta) \in X^{\cS}_\mathcal{W}$ be a higher-order $\h$-extremal with respect to a QDQ approximating cone $\K$ to $\cS$ at the point $(\bar y^0(\bar S), \bar y(\bar S))$. 
We say that this extremal is \emph{normal}
  if, for every choice of multipliers $(p_0, p, \pi, \lambda)$, it holds that $\lambda \ne 0$. Otherwise, it is called an \emph{abnormal higher-order extremal with respect to $\K$}.\footnote{The fact that a higher-order   extremal, as well as  a classical  extremal, is abnormal (with  respect to  some $\K$) does not depend on the cost function $\h$.}
	\end{definition}
	 \begin{remark}{\rm The notion of extremality here depends on the QDQ approximating cone 
$\K$. This concept is slightly more general than the one used in the higher-order Maximum Principle established in \cite{AMRCDC,AMR20}. In particular, we consider a state 
$y$
 evolving on a Riemannian manifold, rather than in a Euclidean space, and we work with QDQ approximating cones, which generalize the Boltyanski approximating cones  utilized in \cite{AMR20}.}\end{remark}
 
 \begin{theo}[Gap and higher-order abnormality]\label{ThIsolated} Let  $(\bar S, \bar w^0, \bar w, \bar y^0, \bar y, \bar\beta)$ be a feasible extended process at which a local infimum gap occurs. Then, for any QDQ approximating cone $\K$ to the target  $\cS$ at  $(\bar y^0(\bar S), \bar y(\bar S))$, the process $(\bar S, \bar w^0, \bar w, \bar y^0, \bar y, \bar\beta)$ is an abnormal higher-order extremal with respect to $\K$.

	\end{theo}
	This result was originally established in \cite{MPR} for extended processes exhibiting a local infimum gap with respect to the $L^\infty$-distance between trajectories. However, extending the proof of \cite[Theorem 1]{MPR} to the setting involving the control distance $\d$ is far from straightforward. It involves nontrivial modifications and a careful re-examination of key arguments. A detailed discussion of this adaptation is deferred to the next section.

	As a straightforward consequence of Theorem \ref{ThIsolated}, we deduce the following sufficient condition for the absence of  a local infimum gap: 
	\begin{theo}[Higher-order normality and no-gap]\label{Thnormality} Let
		$\hat z:=(\hat S,\hat w^0,\hat w ,  \hat y^0,\hat  y,\hat\beta)$ be a   feasible extended   process  which satisfies, for some $r>0$,
		$$
		\h\big((\hat{y}^0,\hat{y})(\hat{S}) \big) \le \inf_{\left\{z=(S,w^0, w ,  y^0,  y,\beta) \in X^{\cS}_{\mathcal{W}^+}: \ \d(z,\hat z)<r\right\} }\h\big((y^0,y)(S)\big).  
		$$
		If $(\hat S,\hat w^0,\hat w ,  \hat y^0,\hat  y,\hat\beta)$ is  a normal higher-order  $\h$-extremal
		for some QDQ approximating cone $\K$ to   $\cS$ at $(\hat y^0(\hat S),\hat y(\hat S))$, then 
		at $(\hat S,\hat w^0,\hat w ,  \hat y^0,\hat  y,\hat\beta)$ there is no local infimum gap.  \end{theo}
		
		\section{Proof of Theorem \ref{ThIsolated}}
		\subsection{Rate independence and set separation}
		Since the structure of the proof closely follows that of \cite[Theorem 3.1]{MPR} --with the main difference being that we now use the topology induced by the control distance~$\d$  rather than the $L^\infty$-distance on trajectories-- we  focus on highlighting the necessary adaptations at the steps where this distinction becomes relevant. The proof of Theorem  \ref{ThIsolated}  is based on two key elements:  
\begin{itemize}
    \item[(i)] the \emph{rate-independence} of the extended control system, which allows us to treat the reference extended process as  an extended process at which there is a local infimum gap with respect to a suitably rescaled extended system with fixed final time; and
    \item[(ii)] a \emph{set-separation argument}, which ensures that the reachable set (via strict-sense trajectories) can be approximated by appropriate ``higher-order'' QDQ approximating cones of the extended reachable set. As noted in  \cite[Subsec. 4.2.5]{MPR}, this is not   straightforward, since the extended reachable set may be significantly larger than the strict-sense one.
\end{itemize}

Let $(\bar S,\bar w^0,\bar w , \bar\alpha, \bar y^0,\bar  y,\bar\beta)\in X^{\cS}_{\mathcal{W}}$ be a feasible extended  process  at which there is a local infimum gap. In case $\bar \beta(\bar S)=K$, the statement  of Theorem \ref{ThIsolated} reduces to the first order conditions  (i)--(v), which follow from   \cite[Theorem 5.2]{PR}. Hence, from now on we suppose $\bar \beta(\bar S)<K$.   

We enlarge the set of extended   processes considered up to now by introducing the larger  set of extended  controls  $\tilde\W\supset\W$,   defined as
	\vsm
	\noindent$ 
	\ds\tilde\W:=
	\bigcup_{S>0}\left( \{S\}\times \Big\{(w^0,w)\in L^\infty ([0,S], \cR_+\times \C): \ \   {\text{ess}\inf} (w^0+|w|)>0\Big\}\right),
	$  
	
	\noindent and the subset $\tilde\W_+:=\{(S,w^0,w)\in\tilde \W: \ w^0>0 \ \text{a.e.}\}\subset \tilde\W$.  Let   
	$X_{\tilde\W}$, $X_{\tilde\W_+}$   denote the set of extended processes $(S,w^0,w, y^0,y,\beta)$, where $(S,w^0,w)\in \tilde\W$ and  $(S,w^0,w)\in \tilde\W_+$, respectively,  while $(y^0,y,\beta)$ is the corresponding solution on $[0,S]$ of  \eqref{extended}.   
 Throughout this section, we will refer to the elements of $X_{\tilde\W}$ as   {\em  extended  processes}, while any element  of the subset  $X_{\tilde\W_+}$ will be called  an {\em embedded strict sense process}. Finally,  we will say that  $(S,w^0,w, y^0,y,\beta)\in X_{\tilde\W}$ is {\em canonical} when $(S,w^0,w)\in\W$, i.e. $w^0+|w|=1$ almost everywhere. With this convention, all the extended  processes considered so far were canonical. 
	
	By  {\em rate-independence} of the extended control system  \eqref{extended} we mean that,
	given any strictly increasing, surjective,  bi-Lipschitz continuous   function $\sigma:[0, S]\to [0, \tilde S]$,  an element  $(\tilde S,\tilde w^0,\tilde w,\tilde y^0,\tilde y,\tilde \beta)$ is an extended process  if  and only if   $(S,w^0,w, y^0,y,\beta)$  given by
	$$
	(w^0,w):=\Big((\tilde w^0,\tilde w)\circ\sigma\Big)\,\frac{d\sigma}{ds}, \quad (y^0,y,\beta):=\left(\tilde y^0,\tilde y,\tilde \beta\right)\circ\sigma
	$$
	is an extended process  (see \cite[Sect. 3]{MR}). If we  call {\it equivalent}  and write
	$
	(\tilde S,\tilde w^0,\tilde w,\tilde y^0,\tilde y,\tilde \beta)  \sim (S,w^0,w,y^0,y,\beta),$
	any two extended  processes as above,\footnote{Actually, $\sim$ is an equivalence relation on the set $X_{\tilde\W}$} we can single out a special representative in any $\sim$equivalence class:
	\begin{definition} Given   $(S,w^0,w, y^0,y,\beta)\in X_{\tilde\W}$, we  set  $[0,S]\ni s\mapsto \sigma(s):=y^0(s)+\beta(s)$,  and define  the  {\rm canonical parameterization} of  $(S,w^0,w, y^0,y,\beta)$,  as
		
		$(S_c,w_c^0,w_c , y_c^0,y_c,\beta_c )
		:=\bigg(\sigma(S),((w^0, w)\circ\sigma^{-1})\cdot \frac{\ d\sigma^{-1}}{ds},(y^0 ,y,  \beta)\circ \sigma^{-1}\bigg). 
		$
	\end{definition}
	Notably,   $w^{0}_c   (s)+\left|w_c (s)\right|=1$ for a.e. $s\in[0,S_c]$, so   $(S_c,w_c^0,w_c , y_c^0,y_c,\beta_c )\in X_{\W}$, i.e. it is a canonical extended  process. One can  easily verify that an extended   process is canonical if and only if it  coincides with its canonical parameterization.

	By considering the enlarged set $\tilde\W$ of extended  controls,  the original control system \eqref{E2} can be    embedded  into the extended system \eqref{extended}, when the latter is thought as defined on the $\sim$quotient space, and the set of strict-sense processes can be identified with  $X_{\tilde\W_+}$.  Precisely,  any strict sense process $(T,u, x,\vz)$ is in one-to-one correspondence with the $\sim$equivalence class $[(S,w^0,w, y^0,y,\beta)]$,  
	where  $(S,w^0,w, y^0,y,\beta)$ is the  canonical embedded strict sense process defined in \eqref{seq}.  Observe that any extended process  $(\tilde S,\tilde w^0,\tilde w ,\tilde y^0,\tilde y,\tilde \beta)$ that is $\sim$equivalent to this  $(S,w^0,w,y^0,y,\beta)$ satisfies $\tilde w^0 > 0$ a.e.. Consequently, it belongs to $X_{\tilde\W_+}$. Moreover, an extended  process is feasible if and only if every equivalent process is feasible, and all equivalent processes yield the same cost.
	 
A crucial step in the proof of Theorem \ref{ThIsolated} consists in showing that, both in the formulation of the extended optimization problem and in the definition of the local infimum gap, we are free to consider --without affecting the validity of the results-- any of the following classes of extended processes:

i) the set of all extended  processes;

ii) the subclass of canonical extended processes (as adopted in the previous sections); or even

iii) any subclass of extended processes $(S,w^0,w,y^0,y,\beta)$ satisfying $R_1 \le w^0 + |w| \le R_2$ almost everywhere, for some constants $0 < R_1 < R_2$.

\vsm	
More in detail, 
 for any pair of extended processes $z_i:=(S_i,w^0_i,w_i,y^0_i,y_i,\beta_i)\in X_{\tilde\W}$,  $i=1,2$, let us introduce the control distance
 $$
 \tilde\d(z_1,z_2):=|S_1-S_2|+\int_0^{S_1\vee S_2}[|w^0_1(s)-w^0_2(s)|+|w_1(s)-w_2(s)|]\,ds,
 $$
where  for any process $(S,w^0,w, y^0,y,\beta)\in X_{\tilde\W}$ we tacitly  assume that $(w^0,w)$ and $( y^0,y,\beta)$ are extended to $\rr_+$ by setting $(w^0,w)(s)=(0,0)$ and $( y^0,y,\beta)(s)=( y^0,y,\beta)(S)$ for all $s\ge S$. It is immediate to see that, if $z_1$, $z_2$ are  canonical processes, then the distances $\d$ and $\tilde\d$ are equivalent, since $\d(z_1,z_2)\le  \tilde\d(z_1,z_2)\le 2\d(z_1,z_2)$.
 
\subsection{Three technical Lemmas}
  \begin{lemma}\label{L0} There exist some $M$, $L>0$ such that, given an extended process  $\bar z:=(\bar S,\bar w^0,\bar w ,  \bar y^0,\bar  y,\bar\beta)\in X_{\tilde\W}$,  for any  $z:=(S,w^0,w,y^0,y,\beta)\in X_{\tilde\W}$  
we have
\begin{equation}\label{est_tilded}
  \begin{array}{l}
  \|y^0-\bar y^0\|_{L^\infty(\rr_+)}\le\int_0^{S\vee \bar S}|w^0 (s)-\bar w^0(s)|\,ds\le  \tilde\d(z,\bar z),  \\[1.0ex]
  \|y -\bar y \|_{L^\infty(\rr_+)}\le Me^{L\bar R}\tilde\d(z,\bar z), \\[1.0ex]
   \|\beta-\bar\beta\|_{L^\infty(\rr_+)}\le\int_0^{S\vee \bar S}|w  (s)-\bar w (s)|\,ds\le  \tilde\d(z,\bar z),
   \end{array}
 \end{equation}
 where   $\bar R:=\bar y^0(\bar S)+\bar\beta(\bar S)=\|\bar w^0\|_{L^1(0,\bar S)}+\|\bar w\|_{L^1(0,\bar S)}$.  
  \end{lemma}
  \begin{proof}  The estimates of $ \|y^0-\bar y^0\|_{L^\infty(\rr_+)}$ and   $\|\beta-\bar\beta\|_{L^\infty(\rr_+)}$ in \eqref{est_tilded} are immediate. 
If $S>\bar S$, for any $s\ge\bar S$, by standard calculations we have 
$$
  \begin{array}{l}
 \ds|y(s) -\bar y(s)|\le \int_0^{\bar S} |f(y(s')w^0 (s')-f(\bar y(s')\bar w^0(s')|\,ds' +\int_{\bar S}^{s}|f(y(s')w^0(s')|\,ds' \\
 \ds  \quad\quad\qquad+ \sum_{i=1}^m \left[\int_0^{\bar S}|g_i(y(s')w^i (s')-g_i(\bar y(s')\bar w^i(s')|\,ds'  
 +\int_{\bar S}^{s}|g_i(y(s')w^i(s')|\,ds'\right]  \\
 \ds \quad\quad\qquad\le  M \int_0^{\bar S}\left[|w^0(s')-\bar w^0(s')|+|w(s')-\bar w(s')|\right]\,ds'\\
\ds  \quad\quad\qquad+L\int_0^{s}|y(s')-\bar y(s')|(\bar w^0(s')+|\bar w(s')|)\,ds'  +
 M  \int_{\bar S}^{S}\left[w^0(s')|+ |w(s')|\right]\,ds', \end{array}
  $$
  where $M$ and $L>0$ depend on the bounds and on the Lipschitz constants of the vector fields $f$, $g_1, \dots,g_m$, respectively. By applying Gronwall Lemma we  obtain
  $$
  \|y -\bar y \|_{L^\infty(\rr_+)}\le Me^{L\bar R}\tilde\d(z,\bar z),
  $$ 
  The proof in case $S<\bar S$ is analogous. 
 \end{proof} 
   
In what follows, we denote by $X^{\cS}_{\tilde\W}$ and $X^{\cS}_{\tilde\W_+}$ the subsets of feasible extended processes within $X_{\tilde\W}$ and $X_{\tilde\W_+}$, respectively. \begin{definition}\label{gapcan} We say that at a process  $\bar z:=(\bar S,\bar w^0,\bar w ,  \bar y^0,\bar  y,\bar\beta)\in X^{\cS}_{\tilde\W}$ {\em there is a local infimum gap  with respect to  $(X^{\cS}_{\tilde\W_+}, \tilde\d)$,} if there exists some $r>0$ such that
$$
\h(\bar y^0(\bar S),\bar  y(\bar S))<\inf_{\left\{z=(S,w^0, w ,   y^0,  y,\beta) \in X^{\cS}_{\tilde\W_+}: \ \tilde\d(z,\bar z)<r\right\}}\h(y^0(S),  y(S)).
$$
\end{definition}
   To avoid confusion, when in  Definition \ref{Dgap} $\hat z$ is a canonical extended process with a local infimum gap, we may say that  the gap is \textit{with respect to} $(X^\cS_{\W_+}, \d)$.

\begin{lemma}\label{L1}  At some $(\bar S,\bar w^0,\bar w ,  \bar y^0,\bar  y,\bar\beta)\in X^{\cS}_{\tilde\W}$   there is a local infimum  gap   with respect to  $(X^{\cS}_{\tilde\W_+}, \tilde\d)$, 
			if and only if at every process
			$(\tilde S,\tilde w^0, \tilde w ,  \tilde y^0,  \tilde y,\tilde \beta)\sim(\bar S,\bar w^0,\bar w ,   \bar y^0,\bar  y,\bar\beta)$   there is a  local infimum  gap  with respect to  $(X^{\cS}_{\tilde\W_+}, \tilde\d)$.
\end{lemma}	
\begin{proof} Let  $\bar z:=(\bar S,\bar w^0,\bar w ,  \bar y^0,\bar  y,\bar\beta)\in X^{\cS}_{\tilde\W}$ be a process at which there is a local infimum gap w.r.t. $(X^{\cS}_{\tilde\W_+}, \tilde\d)$.  Suppose by contradiction that at some  $\tilde z:=(\tilde S,\tilde w^0,\tilde w ,  \tilde y^0,\tilde  y,\tilde\beta)\sim\bar z$  no local infimum gap  w.r.t.   $(X^{\cS}_{\tilde\W_+},\tilde\d)$ occurs. So, there exist a strictly increasing, bi-Lipschitz continuous, onto function  $\sigma:[0,\bar S]\to[0,\tilde S]$, such that $L_1\le \frac{d\sigma}{ds}\le L_2$ a.e. for some $L_2>L_1>0$, and 
$$
\begin{array}{l}
\ds\tilde S=\sigma(\bar S), \quad (\bar w^0,\bar w)(s)=(\tilde w^0,\tilde w)(\sigma(s))\cdot\frac{d\sigma}{ds}(s) \ \ \text{for a.e. $s\in[0,\bar S]$}, \\[1.0ex]
( \bar y^0,\bar  y,\bar\beta)(s)= (\tilde y^0,\tilde  y,\tilde\beta)(\sigma(s)) \ \ \text{for all $s\in[0,\bar S]$},
\end{array}
$$
and a sequence of processes $z_j:=(S_j,w^0_j,w_j,y^0_j,y_j,\beta_j)\in X^{\cS}_{\tilde\W_+}$  such that
$$
\tilde\d(z_j,\tilde z)<\frac{1}{j} \quad\text{for any $j\in\cN$, \ $j\ge1$.}
$$
From Lemma \ref{L0} it follows that  
$$
|S_j-\tilde S|, \ \  \|y_j^0-\tilde y^0\|_{L^\infty(\rr_+)},  \ \  \|\beta-\tilde\beta\|_{L^\infty(\rr_+)}< \frac{1}{j}, 
\quad  \|y_j -\tilde y \|_{L^\infty(\rr_+)}< \frac{C}{j},  
$$
for some $C>0$. As a consequence, we have $\Psi(y^0_j(S_j),y_j(S_j))\to\Psi(\tilde y^0(\tilde S),\tilde y(\tilde S))=\Psi(\bar y^0(\bar S),\bar y(\bar S))$ as $j\to+\infty$.  Possibly passing to a subsequence, we may assume that either
\vsm
\qquad\qquad (a) \  \( S_j \le \tilde{S} \) for every \( j \), \quad or \quad (b)
    \( S_j > \tilde{S} \) for every \( j \).
\vsm 
 \noindent {\em Case} (a). Considering the restriction of the inverse $\sigma^{-1}:[0,S_j]\to[0,\bar S_j]$, for any $j\ge1$ we introduce the feasible, embedded strict-sense process  $\bar z_j:=(\bar S_j,\bar w^0_j,\bar w_j,\bar y^0_j,\bar y_j,\bar \beta_j)\sim z_j$ given by
 $$
 \begin{array}{l}
\ds\bar S_j=\sigma^{-1}(S_j), \quad   (\bar w^0_j,\bar w_j)(s)=(w^0_j,w_j)(\sigma(s))\cdot\frac{d\sigma}{ds}(s) \ \ \text{for a.e. $s\in[0,\bar S_j]$}, \\[1.0ex]
( \bar y^0_j,\bar  y_j,\bar\beta_j)(s)= (y^0_j, y_j,\beta_j)(\sigma(s)) \ \ \text{for all $s\in[0,\bar S_j]$}.
\end{array}
$$
Using the time-change $\sigma=\sigma(s)$, we get
$$
\begin{array}{l}
\ds\tilde\d(\bar z_j,\bar z)=|\bar S_j-\bar S|+\int_0^{\bar S_j}\left[|\bar w^0_j(s)-\bar w^0(s)|+|\bar w_j(s) - \bar w(s)|\right]\,ds \\
\ds\ \quad +\int_{\bar S_j}^{\bar S}\left[\bar w^0(s)+|\bar w(s)|\right]\,ds  \\
\ds  \  = |\sigma^{-1}(S_j)-\sigma^{-1}(\tilde S)|+\int_0^{\bar S_j}| w^0_j(\sigma(s))-\tilde w^0(\sigma(s))| \frac{d\sigma}{ds}(s)\,ds \\
\ds \quad+\int_0^{\bar S_j}|w_j(\sigma(s)) - \tilde w(\sigma(s))|\frac{d\sigma}{ds}(s)\,ds+\int_{\bar S_j}^{\bar S}\left[\tilde w^0(\sigma(s))+|\tilde w(\sigma(s))|\right]\frac{d\sigma}{ds}(s)\,ds  \\
\ds \ \le \left(\frac{1}{L_1}\vee 1\right)\tilde\d(z_j,\tilde z)<\left(\frac{1}{L_1}\vee 1\right)\frac{1}{j}.
\end{array}
$$
Consequently, at $\bar z$ we cannot have a  local infimum gap w.r.t. $(X^{\cS}_{\tilde\W_+}, \tilde\d)$, in contradiction with the hypothesis.

 \noindent {\em Case} (b). For any $j\ge1$,   to introduce  a suitable feasible embedded strict-sense process equivalent to $z_j$, in this case we extend $\sigma^{-1}:[0,\tilde S]\to[0,\bar S]$ by considering   $\sigma_j^{-1}:[0,S_j]\to[0,\bar S_j]$,  given by
$$
\sigma_j^{-1}(\sigma):=\begin{cases}  \sigma^{-1}(\sigma) \ \ &\text{if $\sigma\in[0,\tilde S]$},\\
\bar S+(\sigma-\tilde S)  \ \ &\text{if $\sigma\in[\tilde S, S_j]$},
\end{cases}
$$
where $\bar S_j:=\bar S+(S_j-\tilde S)$. Notice that the inverse $\sigma_j:[0,\bar S_j]\to[0,S_j]$ of this function $\sigma_j^{-1}$ is
$$
\sigma_j (s)=\begin{cases}  \sigma(s) \ \ &\text{if $s\in[0,\bar S]$},\\
\tilde S+(s-\bar S)  \ \ &\text{if $s\in[\bar S, \bar S_j]$}.
\end{cases}
$$
Hence, for any $j\ge1$ we define $\bar z_j:=(\bar S_j,\bar w^0_j,\bar w_j,\bar y^0_j,\bar y_j,\bar \beta_j)\sim z_j$ as
$$
 \begin{array}{l}
\ds\bar S_j=\sigma_j^{-1}(S_j)=\bar S+(S_j-\tilde S),  \\[1.0ex]
 \ds   (\bar w^0_j,\bar w_j)(s)=(w^0_j,w_j)(\sigma_j(s))\cdot\frac{d\sigma_j}{ds}(s) \ \ \text{for a.e. $s\in[0,\bar S_j]$}, \\[1.0ex]
( \bar y^0_j,\bar  y_j,\bar\beta_j)(s)= (y^0_j, y_j,\beta_j)(\sigma_j(s)) \ \ \text{for all $s\in[0,\bar S_j]$}.
\end{array}
$$
Using the time-change $s'=\sigma_j(s)$, we now have
$$
\begin{array}{l}
\ds\tilde\d(\bar z_j,\bar z)=|\bar S_j-\bar S|+\int_0^{\bar S}\left[|\bar w^0_j(s)-\bar w^0(s)|+|\bar w_j(s) - \bar w(s)|\right]\,ds \\
\ds\qquad\qquad\qquad +\int_{\bar S}^{\bar S_j}\left[\bar w_j^0(s)+|\bar w_j(s)|\right]\,ds  \\
\ds\quad\quad = |S_j-\tilde S|+\int_0^{\bar S}\left[| w^0_j(\sigma_j(s))-\tilde w^0(\sigma(s))|+|w_j(\sigma_j(s)) - \tilde w(\sigma_j(s))|\right]\frac{d\sigma_j}{ds}(s)\,ds \\
\ds\qquad\qquad\qquad +\int_{\bar S}^{\bar S_j}\left[w_j^0(\tilde S+(s-\bar S))+|w_j(\tilde S+(s-\bar S))|\right]\,ds \le \tilde\d(z_j,\tilde z)< \frac{1}{j}.
\end{array}
$$
Again, this contradicts the existence  at $\bar z$ of a  local infimum gap w.r.t. $(X^{\cS}_{\tilde\W_+}, \tilde\d)$.
\end{proof}
	
\begin{lemma}\label{L2} At a  feasible canonical extended process $(\bar S,\bar w^0,\bar w ,  \bar y^0,\bar  y,\bar\beta)\in X^\cS_{\W}$   there is a local infimum  gap  w.r.t.  $(X^{\cS}_{\tilde\W_+},\tilde\d)$, 
			 if and only if    there is a  local infimum  gap  w.r.t.  
			$(X^\cS_{\W_+},\d)$, i.e. among canonical processes only.   
\end{lemma}	
\begin{proof}  Ler $\bar z:=(\bar S,\bar w^0,\bar w ,  \bar y^0,\bar  y,\bar\beta)$ be a feasible canonical extended process. 
If at $\bar z$ there is a  local infimum  gap  w.r.t. feasible processes in  $(X^{\cS}_{\tilde\W_+},\tilde\d)$, by definition there are some $r>0$ and $\eps>0$ such that for any feasible embedded strict-sense process $z:= (S,w^0,w, y^0,y,\beta)\in X^{\cS}_{\tilde\W_+}$ such that $\tilde\d(z,\bar z)<r$, we have 
\begin{equation}\label{gap_L2}
	\Psi( \bar y^0(\bar S),\bar  y(\bar S))\le \Psi( y^0(S),y(S))-\eps.
\end{equation}
Now, for any canonical  feasible embedded strict-sense process $z$  with $\d( z,\bar z)<\frac{r}{2}$, we have $\tilde\d( z,\bar z)\le 2\d( z,\bar z)<r$. Therefore \eqref{gap_L2} holds true, and at $\bar z$ there is also a local infimum gap w.r.t.  to
			feasible canonical  processes in $(X^{\cS}_{\W_+},\d)$.
			
Conversely, assume now that  at $\bar z$ there is  a local infimum gap w.r.t.  to feasible canonical  processes in $(X^{\cS}_{\W_+},\d)$ only, but not  w.r.t. the larger set  $(X^{\cS}_{\tilde\W_+},\tilde\d)$. Since $\bar{z}$ is a canonical process, then there exists a sequence  of feasible (possibly non-canonical) embedded strict-sense  processes $z_j:=(S_j,w^0_j,w_j,y^0_j,y_j,\beta_j)\in X^{\cS}_{\tilde\W_+}$, such that
$$
\tilde\d(z_j,\bar z)<\frac{1}{j} \quad\text{for any $j\in\cN$, \ $j\ge1$.}
$$
For any $j\ge1$, let us associate with $z_j$ its canonical representative  $z^c_j:=(S^c_j,w^{0^c}_j,w_j,^c,y^{0^c}_j,y_j^c,\beta_j^c)$. Namely, let us take   $\sigma_j:[0,S_j]\to[0,S_j^c]$ defined as
$$
\begin{array}{l}
\ds\sigma_j(s):=y^0_j(s)+\beta_j(s)=\int_0^s\left[w^0_j(s')+|w_j(s')|\right]\,ds' \ \ \text{for all $s\in[0,S_j]$,} \\
\ds S_j^c:=\sigma_j(S_j)=y^0_j(S_j)+\beta_j(S_j),
\end{array}
$$
and consider the canonical  process  $z^c_j\in X^{\cS}_{\W_+}$  satisfying  
$$
 \begin{array}{l}
\ds (w^0_j,w_j)(s)=(w^{0^c}_j, w_j^c)(\sigma_j(s))\cdot\frac{d\sigma_j}{ds}(s) \ \ \text{for a.e. $s\in[0,S_j]$}, \\[1.0ex]
(y^0_j, y_j,\beta_j)(s)=(y^{0^c}_j,y_j^c,\beta_j^c) (\sigma_j(s)) \ \ \text{for all $s\in[0,S_j]$}.
\end{array}
$$
If  we extend $\sigma_j$ to $\rr_+$ by setting
$$
\sigma_j (s):=\begin{cases}  \sigma_j(s) \ \ &\text{if $s\in[0,S_j]$},\\
S_j^c+(s-S_j)  \ \ &\text{if $s\ge S_j$}, 
\end{cases}
$$
we have that the sequence $(\sigma_j)_j$ is an ${L^\infty(\rr_+)}$-approximation of the identity function $id$. Indeed, recalling that $\bar w^0+|\bar w|=1$  a.e. (because $\bar z$ is canonical),  we have
$$
\begin{array}{l}
\ds\int_0^{+\infty}\left|\frac{d\sigma_j}{ds}(s)-1\right|\,ds= \int_0^{S_j\land\bar S}\left| w^0_j(s)+|w_j(s)|-\bar w^0(s)-|\bar w(s)|\right|\,ds +\\
\ \ \quad\qquad +\begin{cases} 0 \ \ &\text{if $S_j\le \bar S$} \\
\int_{\bar S}^{S_j}\left|w^0_j(s)+|w_j(s)|\right|\,ds+|S_j-\bar S|  \ \ &\text{if $S_j> \bar S$}
\end{cases} \ \le\tilde\d(z_j,\bar z)<\frac{1}{j}.
\end{array}
$$ 
It follows that  $\|\sigma_j-id\|_{L^\infty(\rr_+)}\le\frac{1}{j}$. Let us now estimate $\d(z^c_j,\bar z)$. As in the proof of Lemma \ref{L1}, possibly passing to a subsequence, we may assume that either
\vsm
\qquad\qquad (a) \  \( S^c_j \le \bar{S} \) for every \( j \) \quad or \quad (b)
    \( S^c_j > \bar{S} \) for every \( j \).
\vsm 
 \noindent {\em Case} (a).     We have 
\begin{equation}\label{f1_lemma4.3}
\ds\d(z^c_j,\bar z)=|S^c_j-\bar S|+\int_0^{S^c_j}\left[|w^{0^c}_j(s)-\bar w^0(s)|+|w_j^c(s) - \bar w(s)|\right]\,ds,
\end{equation}
where
$$
\ds|S^c_j-\bar S|\le  |\sigma_j(S_j)-S_j|+|S_j-\bar S|<\frac{1}{j}+\frac{1}{j}=\frac{2}{j},
$$
and, using the time-change $s=\sigma_j(s')$, 
\begin{equation}\label{f2_lemma4.3}
\begin{array}{l}
\ds\int_0^{S^c_j}\left[|w^{0^c}_j(s)-\bar w^0(s)|+|w_j^c(s) - \bar w(s)|\right]\,ds \\[1.0ex]
 \quad\quad\qquad\ds=\int_0^{S_j}\left[|w^{0^c}_j(\sigma_j(s))-\bar w^0(\sigma_j(s))|+|w_j^c(\sigma_j(s)) - \bar w(\sigma_j(s))|\right] \frac{d\sigma_j}{ds}(s)\,ds
 \\[1.0ex]
 \quad\quad\qquad\ds=\int_0^{S_j}\left[\left|w^{0}_j(s)-\bar w^0(\sigma_j(s))\frac{d\sigma_j}{ds}(s)\right|+\left|w_j(s) - \bar w(\sigma_j(s))\frac{d\sigma_j}{ds}(s)\right|\right]\,ds.
\end{array}
\end{equation}   
Now, recalling how $\sigma_j$ approximates the identity and that $|\bar w^0|\le 1$ a.e.,  we have
 $$
\begin{array}{l}
\ds\int_0^{S_j} \left|w^{0}_j(s)-\bar w^0(\sigma_j(s))\frac{d\sigma_j}{ds}(s)\right| \,ds \\[1.0ex]
\quad\qquad \ds\le \int_0^{S_j}\left[|w^{0}_j(s)-\bar w^0(s)|+|\bar w^0(s)-\bar w^0(\sigma_j(s))|\right]\,ds \\[1.0ex]
\quad\qquad \ds+ \int_0^{S_j}\bar w^0(\sigma_j(s))\left|1-\frac{d\sigma_j}{ds}(s)\right|\,ds\le \frac{1}{j} + \int_0^{S_j}|\bar w^0(s)-\bar w^0(\sigma_j(s))|\,ds+\frac{1}{j}.
\end{array}
$$
By standard density results, for every $j\ge1$ there is a continuous function \( \bar w^0_j \in C([0,\bar S+2]) \), which we can assume  verifying $|\bar w^0_j(s)|\le 2$ for all $s$,  such that
$$
\|\bar w^0 - \bar w^0_j\|_{L^1(0,\bar S+2)} < \frac{1}{j}.
$$
Since \( \bar w^0_j \) is uniformly continuous on the compact interval \([0,\bar S+2]\) and  \( \sigma_j \to \mathrm{id} \) uniformly in $\rr_+$ (so, $\sigma_j(\bar S+1)<\bar S+1+ \frac{1}{j}\le \bar S+2$),  there exists \( N\in \mathbb{N} \) such that for all \( j \geq N \),
 \[
\int_0^{\bar S+1} |\bar w^0_j(\sigma_j(s)) - \bar w^0_j(s)| \, ds < \frac{1}{j}.
\]
Recalling that  $S_j<\bar S+\frac{1}{j}\le \bar S+1$, we obtain  
\[
\begin{array}{l}
\ds \int_0^{S_j}|\bar w^0(\sigma_j(s))-\bar w^0(s)|\,ds \leq \int_0^{S_j} |\bar w^0(\sigma_j(s)) - \bar w^0_j(\sigma_j(s))| \, ds \\
\ds\quad\qquad\qquad + \int_0^{S_j} |\bar w^0_j(\sigma_j(s)) -  \bar w^0_j(s)| \, ds + \int_0^{S_j} | \bar w^0_j(s) - \bar w^0(s)| \, ds \\
\ds\qquad\qquad\qquad\qquad\qquad\qquad <   \int_0^{S_j} |\bar w^0(\sigma_j(s)) - \bar w^0_j(\sigma_j(s))| \, ds+\frac{2}{j}.
\end{array}
\] 
To estimate $  \int_0^{S_j} |\bar w^0(\sigma_j(s)) - \bar w^0_j(\sigma_j(s))| \, ds$, we make a change of variables using the fact that each \(\sigma_j \) is bi-Lipschitz. Let \( s' = \sigma_j(s) \), then \( s = \sigma_j^{-1}(s') \), and we get
\[
\begin{array}{l}
\ds\int_0^{S_j} |\bar w^0(\sigma_j(s)) - \bar w^0_j(\sigma_j(s))| \, ds = \int_{0}^{\sigma_j(S_j)} |\bar w^0(s') - \bar w^0_j(s')|\frac{d\sigma_j^{-1}}{ds'}   \, ds' \\
\ds\qquad\qquad\qquad\qquad\le\int_{0}^{\bar S+2} |\bar w^0(s') - \bar w^0_j(s')|\frac{d\sigma_j^{-1}}{ds'}   \, ds' \\
\ds\qquad\qquad\qquad\qquad \le \|\bar w^0 - \bar w^0_j\|_{L^1(0,\bar S+2)}+3\int_0^{+\infty}\left|\frac{d\sigma_j^{-1}}{ds'}(s') -1\right|\,ds'<\frac{4}{j},
\end{array}
\]
since, using the time-change $s=\sigma_j^{-1}(s')$, we have
$$
\begin{array}{l}
\ds\int_0^{+\infty}\left|\frac{d\sigma_j^{-1}}{ds'}(s') -1\right|\,ds'=\int_0^{+\infty}\left|\frac{1}{\frac{d\sigma_j}{ds}(\sigma_j^{-1}(s'))} -1\right|\,ds'  \\
\ds\quad\quad\quad=\int_0^{+\infty}\left|\frac{d\sigma_j}{ds}(\sigma_j^{-1}(s')) -1\right|\frac{d\sigma_j^{-1}}{ds'}(s')\,ds'=\int_0^{+\infty}\left|\frac{d\sigma_j}{ds}(s) -1\right|\,ds<\frac{1}{j}.
\end{array}
$$
Analogous calculations imply that 
$$
 \int_0^{S_j} \left|w_j(s)-\bar w(\sigma_j(s))\frac{d\sigma_j}{ds}(s)\right| \,ds<\frac{4}{j}
 $$
 Hence,  $\ds\d(z^c_j,\bar z)< \frac{10}{j}$, in contradiction with the hypothesis  that at $\bar z$ there is a local infimum gap w.r.t. canonical embedded strict-sense processes. Hence, in view of \eqref{f1_lemma4.3}, \eqref{f2_lemma4.3},  the proof in case~(a) is concluded.

In case~(b), the proof follows from that of case~(a), since  
$$
\begin{array}{l}
\ds\d(z^c_j,\bar z)=|S^c_j-\bar S|+\int_0^{\bar S}\left[|w^{0^c}_j(s)-\bar w^0(s)|+|w_j^c(s) - \bar w(s)|\right]\,ds \\
\ds \ \ \quad\qquad\le |S^c_j-\bar S|+\int_0^{S^c_j}\left[|w^{0^c}_j(s)-\bar w^0(s)|+|w_j^c(s) - \bar w(s)|\right]\,ds.
\end{array}
$$  
\end{proof}
\subsection{Conclusion of the proof of Theorem \ref{ThIsolated}}
The three  lemmas of the previous subsection  justify the fact that it is not restrictive to consider a local infimum gap only at canonical processes and  solely among canonical processes. As a  main consequence one has  that reference canonical extended process $\bar z=(\bar S,\bar w^0,\bar w,  \bar y^0,\bar  y,\bar\beta)$ exhibits a local infimum gap even among the non-canonical processes in $(X^{\cS}_{\tilde\W_+},\tilde\d)$. 

	 For every    $r>0$,  let us now  define the  following sets:   
	  $$
		\begin{array}{l}
			\mathcal{R'}^{r}_{\tilde\W_+} := \Big\{(y^0,y,\beta)(S): \   z=(S, w^0,w,y^0,y,\beta)\in X_{\tilde\W_+}, 
			\ \tilde\d \left(z, \bar z\right)< r  \Big\},
			\\[1.0ex]
			\mathcal{R'}_{\tilde\W}^{r}:=  \Big\{(y^0,y,\beta)(S): \   z=(S, w^0,w,y^0,y,\beta)\in X_{\tilde\W}, 
			\ \tilde\d  \left(z, \bar z\right)< r  \Big\},
		\end{array}
		$$ 
as the   \textit{reachable set} and the \textit{extended  reachable set} (of radius $r$), respectively.
	
From this point onward, the proof proceeds in a manner entirely analogous to the proof of   \cite[Theorem~3.1]{MPR}, and is therefore omitted. In particular, we note that the needle variations and bracket-like variations introduced in \cite{MPR} yield approximations of the reference control $(\bar w^0,\bar w)$ that converge to it in the control distance~$\tilde{\d}$. Thanks to this property, they allow for the construction of a $QDQ$ approximating cone to~${\mathcal{R}'}^{r}_{\tilde\W_+}$ at $(\bar{y}^0(\bar{S}), \bar{y}(\bar{S}))$.




%
%
%
%
%
%
%

\end{document}